\documentclass[12pt]{article}
\usepackage{amssymb}
\usepackage{amsthm}
\usepackage{amsmath}
\usepackage[dvips]{graphicx}

\newtheorem{defi}{Definition}
\newtheorem{teo}{Theorem}
\newtheorem{lem}{Lemma}

\newcommand{\ve}{\varepsilon}
\newcommand{\vp}{\varphi}

\newcommand{\pa}{\partial\,}
\newcommand{\ol}{\overline}

\newcommand{\om}{\omega}
\newcommand{\be}{\begin{equation}}
\newcommand{\ee}{\end{equation}}
\newcommand{\const}{\operatorname{const}}
\newcommand{\ipd}{\stackrel{\normalfont\text{def}}{=}}
\newcommand{\dpsi}{\dot{\psi}_0}
\newcommand{\os}{\overline{\sigma}_{ln}}
\newcommand{\ts}{\tilde{\sigma}_{ln}}

\usepackage[usenames]{color}

\begin{document}
\allowdisplaybreaks

\title{Interaction of 3 solitons   for the GKdV-4  equation }

\author{ Georgy~A.~Omel'yanov\thanks{
University of Sonora, Rosales y Encinas s/n, 83000, Hermosillo,
Sonora, Mexico,\ omel@hades.mat.uson.mx} }
\date{}
\maketitle
\begin{abstract}
We describe  an approach  to construct multi-soliton asymptotic solutions  for non-integrable equations.
The general idea is realized in the case of three waves and for  the  KdV-type equation with nonlinearity $u^4$.
A brief review of asymptotic methods as well as results of numerical simulation are included.
\end{abstract}


\section{Introduction}

\subsection{Statement of the problem}

 It is well known that an arbitrary number of solitary waves collide for integrable non-linear equations in an enormous manner:
 they pass through  each other almost as  linear waves. The aim of this paper is the consideration:
 can we believe that such type of interaction is conserved (in a sense) for some essentially non-integrable equations?

 Needless to recall that the integrability implies both the possibility to find exact solutions of complicated structures
 and that the equation has some special properties. Conversely, the non-integrability implies that we do not have,
 at least nowadays,  neither explicit solutions, nor special useful properties of the problem.

 We consider the general techniques and averaging method  by
a specific but typical  example of the Generalized Korteweg-de Vries-4 equation with
small dispersion, that is:
\begin{equation}\label{1}
\frac{\pa u}{\pa t}+\frac{\pa u^m}{\pa x}+\varepsilon^2
\frac{\partial^3u}{\pa x^3}= 0, \quad m=4,\quad x \in
\mathbb{R}^1, \quad t
> 0,
\end{equation}
where $\varepsilon<<1$ is a small parameter.

The GKDV family  (\ref{1}) contains integrable
 ($m=2$ and $m=3$, that is the KdV and MKdV equations) and essentially non-integrable ($m\geq4$)
 equations. The last means that   there is not  any method at present to construct exact
 solutions  of the Cauchy problem with more or less
general initial data.

More in detail, there is known that (\ref{1}) for  $m\geq 6$ is unstable, whereas the
case  $m=5$ is conditionally stable (unstable for solitons) \cite{Bona1, Merle}.
As for $m=4$, this case is stable and there are known some exact particular
solutions including  the solitary wave:
\begin{equation}  u=A\om\left(\beta
\frac{x-Vt}{\varepsilon}\right),\quad\om(\eta)=c\cosh^{-2/3}(\eta),\quad A =
\frac{1}{\gamma}\beta^{2/3},\label{2}
\end{equation}
where $\beta>0$ is an arbitrary number, $c$ is such that
\begin{equation}\label{4}
\int_{-\infty}^\infty \omega (\eta)d\eta=1,
\end{equation}
and
\begin{equation}\label{3}  \quad V =
\frac{a_4}{\gamma^3}\beta^2.
\end{equation}
Here and in what follows we use the notation
\begin{equation}\label{3a}
a_k=\int_{-\infty}^\infty\omega^k(\eta)\,d\eta,\quad k\geq1,\quad
a'_2=\int_{-\infty}^\infty\Big(\frac{d\omega}{d\eta}\Big)^2\,d\eta
\end{equation}
 and the identities
\begin{equation}\label{3b}
\gamma=\Big(\frac37\frac{a_2a_4}{a'_2}\Big)^{1/3},\quad 2a_2a_4=\frac75a_5.
\end{equation}

In  view of the general wave propagation theory,  there appear questions both about
the stability of the solitary wave solution (\ref{2}) with respect
to small perturbations of the equation, and about the character of
the solitary wave collision.

\subsection{Prehistory: single-phase asymptotic solutions}
 The non-integrability implies the use of asymptotic approaches. We consider waves
of arbitrary amplitude (of the value $O(1)$), treating the dispersion $\ve$ as a small parameter.
Thus, there appear "fast" $x/\ve$, $t/\ve$ and "slow" $x$, $t$  variables. Of course, it is possible
to rescale $\eta=x/\ve$, $\tau=t/\ve$ and pass to  "fast" $\eta$, $\tau$ and "slow" $\ve\eta$, $\ve\tau$ variables.
However we prefer the first version.

The modern asymptotic technique, which is based on ideas  by Poincare,  van der Pol,
  Krylov-Bogoliubov and others, has been created  firstly by G.E. Kuzmak (ODE, 1959, \cite{Kuzmak})
  and G.B. Whitham, (PDE, 1965, \cite{Whitham1, Whitham2}, see also \cite{Whitham3}) for rapidly
  oscillating asymptotic solutions of non-linear equations.
The famous Whitham method deals with
a Lagrangian formulation and allows to find slowly  varying amplitude and wave
number of non-uniform wave trains. This approach determined  the development of the non-linear
perturbation theory in 1970s. At the same time, the passage from the original equation to the Lagrangian
 seemed to be artificial. For this reason J.C. Luke (1966, \cite{Luke})
created a version of the Whitham method, which allows to construct asymptotic solutions with arbitrary
precision appealing directly to the original equation. More in detail, for the equation $L(u,\ve u_t, \ve u_x, \dots)=0$ we write the ansatz
\begin{equation}\label{5}
u=Y_0(\tau,t,x)+\ve Y_1(\tau,t,x)+\dots,
\end{equation}
where $\tau=S(x,t)/\ve$, $Y_k(\tau+T,t,x)=Y_k(\tau,t,x)$, $T=\const$, and $S(x,t)$, $Y_k(\tau,t,x)$
are arbitrary functions from $ C^\infty$. Since
$$\ve\pa_tY_k\big(S(x,t)/\ve,t,x\big)=\{S_t\pa_\tau Y_k(\tau,t,x)+\ve\pa_tY_k(\tau,t,x)\}|_{\tau=S/\ve},$$
we obtain the chain of  ordinary (with respect to $\tau$) equations, the first of them is non-linear,
\begin{equation}\label{5a}
L\big(Y_0(\tau,t,x),S_t\pa_\tau Y_0(\tau,t,x), S_x\pa_\tau Y_0(\tau,t,x), \dots\big)=0,
\end{equation}
and others are non-homogenous linearization:
\begin{align}
L'\big(Y_0(\tau,t,x),S_t\pa_\tau ,& S_x\pa_\tau , \dots\big)Y_k(\tau,t,x)\notag\\
&=F_k\big(Y_0(\tau,t,x),\dots,Y_{k-1}(\tau,t,x)\big),\quad k\geq 1.\label{5b}
\end{align}
It is assumed that (\ref{5a}) has a $T$-periodic solution. Then there appear the orthogonality conditions
\begin{equation}\label{5c}
\int_0^TF_k\big(Y_0(\tau,t,x),\dots,Y_{k-1}(\tau,t,x)\big)Z_id\tau=0,\quad i=1,\dots,l,
\end{equation}
which guarantee the solvability of  (\ref{5b}) in the space of $T$-periodic  smooth bounded functions.
Here $\{Z_i, i=1,\dots,l\}$ is the kernel of the operator adjoint to $L'$. Moreover, (\ref{5c}) allow to
define the phase $S(x,t)$ and all the "constants" of integration of the equations (\ref{5a}), (\ref{5b}).

It seemed that the same procedure can be used to construct a perturbed soliton-type solution (with trivial alterations).
However, it is not true, and a mechanical repetition of the Whitham construction leads to some "paradoxes"
and senseless solutions (see, for example, \cite{Scott}, pp. 303 - 306). The situation has been improved by
V. Maslov and G. Omel'yanov (1981, \cite{MasOm1}, see also \cite{MasOm2}). A little bit later a similar construction
has been developed by I. Molotkov and S. Vakulenko (see e.g. \cite{MolVak}). To illustrate the modification \cite{MasOm1} let us consider
 the perturbed GKdV-4 equation (\ref{1}),
\begin{equation}\label{5d}
\frac{\pa u}{\pa t}+\frac{\pa u^4}{\pa x}+\varepsilon^2
\frac{\partial^3u}{\pa x^3}= R,
\end{equation}
where $R=R(x,t,u,\ve u_x, \ve^2 u_{xx},\dots)$ is  "small" in our scaling and $R|_{u=0}=0$.

To find a self-similar soliton-type asymptotics
 we restrict the soliton part of the solution on the zero-level set of the
 phase $S(x,t)=x-\vp(t)+O((x-\vp(t))^2)$. This allows to avoid the appearance of
 some nonuniqueness effects (see {\cite{MasOm2}, pp. 24 - 26). Next we take into account that a smooth small "tail" can
 appear after the soliton.
 Therefore, instead of  (\ref{5}) we write the ansatz in the form:
\begin{equation}\label{5e}
u=Y_0(\tau,t)+\ve Y_1(\tau,t,x)+\dots,
\end{equation}
where $\tau=\big(x-\vp(t)\big)/\ve$, $Y_k$ are smooth bounded function such that $Y_0(\tau,t,x)$ tends to 0 as
$\tau\to\pm\infty$,  $Y_k(\tau,t,x)\to0$ as $\tau\to+\infty$,
and $Y_k(\tau,t,x)\to Y_k^-(x,t)$ as $\tau\to-\infty$ for $k\geq1$, and $\vp$  belongs to  $ C^\infty$.

Similar to (\ref{5a}), substituting (\ref{5e}) into (\ref{5d}) we obtain  the nonlinear model equation
\begin{equation}\label{6}
-\vp_t\frac{d Y_0}{d\tau}+\frac{d Y_0^4}{d\tau}+
\frac{d^3 Y_0}{d\tau^3}= 0
\end{equation}
 and define the shape of
the leading term in (\ref{5e}),
\begin{equation}\label{10}
Y_0=A(t)\om\big(\beta\tau+\vp_1(t)\big),\quad  \beta^{2/3}(t)=\gamma A(t),
\end{equation}
 as well as the similar
to (\ref{3}) relation
\begin{equation}\label{11}
\frac{d\vp}{dt} = a_4 A^3(t).
\end{equation}
 Here $\vp_1(t)$ is a "constant" of integration. Next to find the deficient relation between $\vp$ and $A$ we consider the first
 correction $Y_1$ freezed on the soliton front $x=\vp(t)$. Denoting $\check{Y}_1(\tau,t)=Y_1(\tau,t,x)|_{x=\vp(t)}$, we pass to the equation:
\begin{equation}\label{7}
\frac{d }{d\tau}\Big\{-\vp_t\check{Y}_1+4 Y_0^3\check{Y}_1+
\frac{d^2 \check{Y}_1}{d\tau^2}\Big\}= R\big(\vp,t,Y_0,Y_{0\tau},Y_{0\tau\tau},\dots\big)-Y_{0\,t}.
\end{equation}
Respectively, to guarantee the existence of the desired correction
$\check{Y}_1$, we obtain the following conditions:
\begin{align}
 &\frac{d}{dt}\int_{-\infty}^\infty Y_0^2d\tau =2\int_{-\infty}^\infty Y_0R\big(\vp,t,Y_0,Y_{0\tau},Y_{0\tau\tau},\dots\big)d\tau,\label{8}\\
 &\vp_t\check{Y}_1|_{\tau\to-\infty}=\int_{-\infty}^\infty \Big\{R\big(\vp,t,Y_0,Y_{0\tau},Y_{0\tau\tau},\dots\big)-Y_{0\,t}\Big\}d\tau.\label{9}
\end{align}
Calculating the integrals  in (\ref{8}), we complete  (\ref{11}) by the equation
\begin{equation}\label{11a}
a_2\frac{d }{dt}\frac{A^2}{\beta}=
2\frac{A}{\beta}\mathfrak{R},
\ee
where
\be
\mathfrak{R}=\int_{-\infty}^\infty \om(\eta)R\big(\vp,t,A\om(\eta),A\beta\om(\eta)_{\eta},A\beta^2\om(\eta)_{\eta\eta},\dots\big)d\eta.
\label{11aa}
\end{equation}
This allows us to determine the phase and amplitude dynamics. The equations (\ref{11}), (\ref{11a}) have
been called "Hugoniot-type conditions" \cite{MasOm1} since they do not depend on $\ve$, whereas  the solitary wave (\ref{10}) $Y_0\big((x-\vp(t))/\ve,t\big)$
disappears (in $\mathcal{D}'$ sense) as $\ve\to0$. Let us recall that the Rankine-Huginiot conditions remain the same both for parabolic regularization of shock waves,
and for the limiting non-smooth solutions.

Furthermore, returning to the asymptotic construction, we note that (\ref{9}) implies the equality
\begin{equation}\label{11b}
a_4A^3\check{Y}_1^-=\frac{1}{\beta}
 \int_{-\infty}^\infty R\big(\vp,t,A\om(\eta),A\beta\om(\eta)_{\eta},A\beta^2\om(\eta)_{\eta\eta},\dots\big)d\eta
 -\frac{d}{dt}\frac A\beta,
\end{equation}
where $\check{Y}_1^-=\check{Y}_1|_{\tau\to-\infty}$. Now we integrate the equation  (\ref{7}) and
find the structure of the first freezed correction
\begin{equation}\label{14}
\check{Y}_1(\tau,t)=\check{Y}_1^-(t)\chi(\tau,t)+Z_1(\tau,t)+c_1(t)Y_{0\tau}'(\tau,t),
\end{equation}
where $\chi$ and $Z_1$ are some fixed functions such that
\begin{align}  &Z_1\to0 \quad \text{as}\quad \tau\to\pm\infty,\notag\\
&\chi\to0 \quad \text{as}\quad \tau\to+\infty,\quad \chi\to 1
\quad \text{as}\quad \tau\to-\infty,\notag
\end{align}
and $c_1$ is an arbitrary "constant" of integration.

The next step of the construction is the extension of $\check{Y}_1(\tau,t)$ to $Y_1(\tau,t,x)$ in the following manner:
\begin{equation}\label{14a}
Y_1(\tau,t,x)=u_1^-(t,x)\chi(\tau,t)+Z_1(\tau,t)+c_1(t)Y_{0\tau}'(\tau,t),
\end{equation}
where $u_1^-$ is a smooth function such that
\begin{align}  &\frac{\pa u_1^-}{\pa t}=u_1^-R'_u(x,t,0,\dots),\quad x<\vp(t),\quad t>0,\label{15}\\
&u_1^-|_{x=\vp(t)}=\check{Y}_1^-,\quad t>0.\label{16}
\end{align}
Continuing the procedure we can easily construct the one-phase
self-similar asymptotic solution with arbitrary precision.

Let us note finally that self-similarity implies the special choice of the initial data. In particular,
the initial function $Y_1(\tau,0,x)$
should be of the special form (\ref{14a}) with arbitrary $c_1(0)$
and arbitrary $u_1^-(x,0)$ under the condition
\begin{equation}\label{17}
u_1^-(x,0)|_{x=\vp(0)}=\check{Y}_1^-(0).
\end{equation}
If it is violated and, for example, $u|_{t=0}=A(0)\om\big(\beta
(x-\vp(0))/\varepsilon\big)$, then the perturbed soliton generates a rapidly oscillating tail of the
amplitude $o(1)$ (the so called "radiation") instead of the smooth tail $\ve u_1^-(x,t)$
(see \cite{Kal} for the perturbed KdV equation). However,
$\ve u_1^-(x,t)$ describes sufficiently well the tendency of the radiation amplitude behavior (see e.g. \cite{OmVG}).

\subsection{History: two-phase asymptotic solution}

Concerning the solitary wave collision, this problem is much more
complicated. Indeed, to describe the interaction of two waves of
the form (\ref{2}), we should look for the asymptotics as a
two-phase function,
\begin{equation}\label{18}
u=W\Big(\frac{x-\vp_1}{\ve},\frac{x-\vp_2}{\ve},t\Big)+o(1),
\end{equation}
where $W(\tau_1,\tau_2,t)$ has properties similar to the
two-soliton solution of the KdV equation. However, to construct
$W$ we obtain a non-linear PDE, which is, in fact, equivalent to
the original GKdV-4 equation  (\ref{1}). So, the existence of such
asymptotics remains unknown. The same is true for any essentially
non-integrable equation. Respectively, there is not any
possibility to construct  a classical
asymptotic solution (that is, with the remainder in the
$C$-sense).

At the same time it is easy to note that the solitary wave
solutions (soliton or kink type) tend to distributions as
$\ve\to0$. This allows to treat the equation in the weak sense and, respectively,
look for singularities instead of regular functions. Obviously, non-integrability
implies that we cannot find neither classical nor weak exact solutions.
However, we can construct an asymptotic weak solution considering the smallness of  the remainder in the weak sense.

Originally, such idea had been suggested by V. Danilov and V. Shelkovich
for shock wave type solutions (1997, \cite{DanShel1}), and after that it has
been developed and adapted  for
many other problems (V. Danilov, G. Omel'yanov, V. Shelkovich, D. Mitrovic and others,
\cite{DanShel2} - \cite{Xiu} and references therein). We called this approach the "weak asymptotics method".

For the special case of soliton-type solutions we note now that  they have the value $O(\ve)$
in the weak sense. Thus, the remainder for the leading term of the asymptotic solution should be $O(\ve^2)$
in the weak sense. However,  the GKdV equations (\ref{1}) degenerate to a first-order PDE in $\mathcal{D}'$ for this precision.
The same fact has been noted by Danilov, Omel'yanov, and
 Radkevich (1997, \cite{DanOmRad}) by the consideration of a free boundary
problem. There has been suggested also a way about how to overcome this obstacle.  Applying these ideas  to the equation (\ref{1}), we pass to
the following definition of the weak asymptotic solution \cite{DanOm}:

\begin{defi}\label{def2.1}
A sequence $u(t, x, \varepsilon )$, belonging to
$\mathcal{C}^\infty (0, T; \mathcal{C}^\infty (\mathbb{R}_x^1))$
for $\varepsilon =\const> 0$ and belonging to $\mathcal{C} (0, T;
\mathcal{D}' (\mathbb{R}_x^1))$ uniformly in $\varepsilon\geq0$, is
called a weak asymptotic mod $ O_{\mathcal{D}'}(\varepsilon^2)$
solution of (\ref{1}) if the relations
\begin{equation}\label{19}
\frac{d}{dt} \int_{-\infty}^\infty u\psi dx -
\int_{-\infty}^\infty  u^4 \frac{\pa \psi}{\pa x} dx = O
(\varepsilon^2),
\end{equation}
\begin{equation}\label{20}
\frac{d}{dt} \int_{-\infty}^\infty u^2 \psi dx
-\frac{8}{5}\int_{-\infty}^\infty u^{5}\frac{\pa \psi}{\pa x}dx+
3\int_{-\infty}^\infty \left( \varepsilon\frac{\pa u}{\pa x}
\right)^2 \frac{\pa \psi}{\pa x}dx = O (\varepsilon^2)
\end{equation}
hold uniformly in $t$ for any test function $\psi = \psi(x) \in
\mathcal{D} (\mathbb{R}^1)$.
\end{defi}
Here the right-hand sides are $\mathcal{C}^\infty$-functions for
$\varepsilon=\const > 0$ and  piecewise continuous functions
uniformly in $\varepsilon \geq 0$. The estimates are understood in
the $\mathcal{C}(0, T)$ sense:
$$g(t, \varepsilon) = O (\varepsilon^k) \leftrightarrow \max_{t\in[0,T]} |g(t,
\varepsilon)| \leq c\varepsilon^k.$$

\begin{defi} \label{def2.2}
A function  $v(t, x, \varepsilon)$ is said to be of the value $
O_{\mathcal{D}'}(\varepsilon^k)$ if the relation
$$ \int_{-\infty}^\infty v(t, x, \varepsilon )\psi(x) dx = O(\varepsilon^k) $$
holds uniformly in $t$ for any test function $\psi \in \mathcal{D}
(\mathbb{R}_x^1)$.
\end{defi}
The sense of the relation (\ref{19}) is obvious: it is the
adaptation of the standard $\mathcal{D}'$-definition to asymptotic
mod $ O_{\mathcal{D}'}(\varepsilon^2)$ solution which  belongs to
$\mathcal{C} (0, T; \mathcal{D}' (\mathbb{R}_x^1))$. Next we note again
 that (\ref{19}) cannot be a unique satisfactory
condition since here has been lost the difference between the
GKdV-4 equation and the limiting first order equation (with
$\ve=0$). To involve the dispersion term into the consideration, we supplement (\ref{19})
by the additional condition (\ref{20}). It can be treated as a version of (\ref{19}) but
 for special test functions $u\,\psi(x)$, $\psi \in
\mathcal{D} (\mathbb{R}_x^1)$, which vary rapidly together with  the
solution.

It is important also that (\ref{20}) duplicates the orthogonality
condition which appears for single-phase asymptotics. Indeed, the
adaptation of the Definition 1 to the perturbed equation
(\ref{5d}) implies the transformation of (\ref{20}) to the
following form:
\begin{align}
\frac{d}{dt} \int_{-\infty}^\infty u^2 \psi dx
-\frac{8}{5}\int_{-\infty}^\infty u^{5}\frac{\pa \psi}{\pa x}dx&+
3\int_{-\infty}^\infty \left( \varepsilon\frac{\pa u}{\pa x}
\right)^2 \frac{\pa \psi}{\pa x}dx
\label{21}\\
&-2\int_{-\infty}^\infty u R \psi dx= O
(\varepsilon^2).\notag
\end{align}
Next for $u$ of the form (\ref{5e}), (\ref{10}) we calculate the weak expansion:
\begin{align}
\int_{-\infty}^\infty u^k(x,t)\psi(x)dx=\ve\frac{A^k}{\beta}\int_{-\infty}^\infty \om^k(\eta)&\psi(\vp+\ve\frac\eta\beta)dx
+O(\ve^2)\notag\\&=\ve a_k\frac{A^k}{\beta}\psi(\vp)
+O(\ve^2), \notag
\end{align}
where $k\geq1$. Thus
\be
u^k=\ve a_k\frac{A^k}{\beta}\delta(x-\vp)+O_{\mathcal{D}'}(\ve^2).\label{21b}
\ee
In the same manner we obtain
\begin{align}
&(\ve u_x)^2=\ve a'_2\beta A^2\delta(x-\vp)+O_{\mathcal{D}'}(\ve^2),\label{21a}\\
&u R(x,t,u,\ve u_x, \ve^2 u_{xx},\dots)=\ve\frac A\beta \mathfrak{R}\delta(x-\vp)+O_{\mathcal{D}'}(\ve^2),\label{21c}
\end{align}
where $\mathfrak{R}$ has been defined in (\ref{11aa}).
Substitution of (\ref{21b}) - (\ref{21c}) into (\ref{21}) implies the relation
\begin{align}
\ve \Big\{-a_2\frac{A^2}{\beta}\frac{d\vp}{dt}+a_5\frac85\frac{A^5}{\beta}-3a'_2\beta A^2\Big\}&\delta'(x-\vp)
\notag\\+\ve \Big\{a_2\frac{d}{dt}\frac{A^2}{\beta}-2\frac A\beta \mathfrak{R}\Big\}&\delta(x-\vp)=
O_{\mathcal{D}'}(\varepsilon^2).\label{21d}
\end{align}
Since $\delta(x-\vp)$ and $\delta'(x-\vp)$ are linearly independent, their coefficients in
 (\ref{21d}) should be equal to zero. Taking into account the identities (\ref{3}) we obtain again the equations
(\ref{11}), (\ref{11a})  for the one-phase asymptotics (\ref{5e}).

Let us revert to the two-wave interaction. Following \cite{DanOm} (see also \cite{DanOmShel}), we present the ansatz
 as the sum of two distorted solitons (\ref{2}), that is:
\begin{equation}\label{22}
u=\sum_{i=1}^2G_i\om\left(\beta_i
\frac{x-\vp_i}{\varepsilon}\right),
\end{equation}
where
\be G_i=A_i+S_i(\tau),\quad\vp_i=\vp_{i0}(t)+\ve\vp_{i1}(\tau),\quad\tau=\beta_1\big(\vp_{20}(t)-\vp_{10}(t)\big)/\ve,\label{22a}\ee
$A_i$  are the original amplitudes and $\vp_{i0}=V_it+x_{i0}$ describe the trajectories of the non-interacting  waves
(\ref{2}), $\beta_i=(\gamma A_i)^{2/3}$. We assume that $A_1<A_2$ and $x_{10}>>x_{20}$, therefore, the trajectories $x=\vp_{10}$ and
$x=\vp_{20}$ intersect at a point $(x^*,t^*)$. Next we define the "fast time"
$\tau$ to characterize
the distance between the trajectories  $\vp_{i0}$ and we assume that $S_i(\tau)$, $\vp_{i1}(\tau)$ are such that
\begin{align}
&S_i\to0 \quad \text{as}\quad \tau\to\pm\infty,\label{23}\\
&\vp_{i1}\to0 \quad \text{as}\quad \tau\to-\infty,\quad
\vp_{i1}\to \vp_{i1}^\infty=\const_i \quad \text{as}\quad
\tau\to+\infty.\label{24}
\end{align}
It is obvious that the existence of the weak asymptotics  (\ref{22}) with the properties (\ref{23}), (\ref{24})
implies that the solitary waves (\ref{2}) interact like the KdV solitons at least in the leading term.

To construct the asymptotics we should calculate again the weak expansions for $u^k$ and $(\ve u_x)^2$.
It is easy to check that
\be
u=\ve \sum_{i=1}^2\frac{G_i}{\beta_i}\delta(x-\vp_i)+O_{\mathcal{D}'}(\ve^3).\label{23aa}
\ee
At the same time
\begin{align}
&\int_{-\infty}^\infty u^2(x,t)\psi(x)dx=\ve\sum_{i=1}^2\frac{G_i^2}{\beta_i}\int_{-\infty}^\infty \om^2(\eta)\psi(\vp_i+\ve\frac\eta\beta_i)dx\label{23a}\\&
+2G_1G_2\int_{-\infty}^\infty \om\left(\beta_1\frac{x-\vp_1}{\ve}\right)\om\left(\beta_2\frac{x-\vp_2}{\ve}\right)\psi(x)dx.\label{23b}
\end{align}
We take into account that the integrand in (\ref{23b}) vanishes exponentially fast as $|\vp_1-\vp_2|$  grows,
thus, the main contribution gives the point $x^*$.  We write
\be
\vp_{i0}=x^*+V_i(t-t^*)=x^*+\ve \frac{V_i}{\beta_1(V_2-V_1)}\tau \quad \text{and} \quad \vp_{i}=x^*+\ve\chi_{i},\label{23c}
\ee
where $\chi_i=V_i\tau/\big(\beta_1(V_2-V_1)\big) + \vp_{i1}.$ Next we transform the integral in (\ref{23b}) to the following form:
\be
\frac\ve{\beta_2}\int_{-\infty}^\infty  \om(\theta_{12}\eta-\sigma_{12})\om(\eta)\psi\big(x^*+\ve\chi_2+\ve\frac\eta\beta_2\big)d\eta,\label{23d}
\ee
where $\theta_{12}=\beta_1/\beta_2$, $\sigma_{12}=\beta_1(\vp_1-\vp_2)/\ve$.
It remains to apply the formula
\be
f(\tau)\delta(x-\vp_i)=f(\tau)\delta(x-x^*)-\ve\chi_if(\tau)\delta'(x-x^*)+O_{\mathcal{D}'}(\ve^2),\label{23e}
\ee
which holds for each $\vp_i$ of the form (\ref{23c}) with slowly increasing $\chi_i$ and for $f(\tau)$ from the Schwartz space.
 Moreover, the second term in (\ref{23e}) is $O_{\mathcal{D}'}(\ve)$. Thus, under the assumptions (\ref{23})
  we can modify (\ref{23aa})-(\ref{23b})
to the final form:
\begin{equation}\label{25}
u=\ve\sum_{i=1}^2\frac{A_i}{\beta_i}\delta(x-\vp_{i})+\ve\sum_{i=1}^2\frac{S_i}{\beta_i}\big\{\delta(x-x^*)-\ve\chi_i\delta'(x-x^*)\big\}
+O_{\mathcal{D}'}(\varepsilon^3),
\end{equation}
\begin{align}
u^2&=\ve a_2\sum_{i=1}^2\frac{A_i^2}{\beta_i}\delta(x-\vp_{i})+\ve
a_2\Big\{
\sum_{i=1}^2\frac{1}{\beta_i}(2A_iS_i+S_i^2)\notag\\
&+2\frac{G_1G_2}{\beta_2}\lambda_{2,1}(\sigma_{12})\Big\}\delta(x-x^*)
+O_{\mathcal{D}'}(\varepsilon^2),\label{26}
\end{align}
where  the convolution $\lambda_{2,1}(\sigma_{12})$ describes the product of
two waves. In view of further applications we present such type of convolutions in the general version:
\begin{equation}\label{27}
\lambda_{m,k}^{(j)}(\sigma_{ln})=\frac{1}{a_m}\int_{-\infty}^\infty
\eta^j\om^{m-k}(\eta_{ln})\om^k(\eta)d\eta,\; \eta_{ln}\ipd\theta_{ln}\eta-\sigma_{ln},\; \theta_{ln}\ipd\frac{\beta_l}{\beta_n},
\end{equation}
where $1\leq k<m$, $m\geq2$,  $j=0$ or $j=1$, and we write $\lambda_{m,k}(\sigma_{ln})\ipd\lambda_{m,k}^{(0)}(\sigma_{ln})$ simplifying the notation.

To  calculate the time-derivative of $u$ with the accuracy $O_{\mathcal{D}'}(\varepsilon^2)$ it is enough to use
the expansion (\ref{25}), the assumptions  (\ref{23}), (\ref{24}), and to apply the formula (\ref{23e}) again. Thus,
\begin{align}
\frac{\pa u}{\pa t}&=\sum_{i=1}^2\frac{\dot{\psi}_0}{\beta_i}\frac{dS_i}{d\tau}\delta(x-x^*)-\ve\sum_{i=1}^2V_i\frac{A_i}{\beta_i}\delta'(x-\vp_{i})\notag\\
&-\ve\dpsi\frac{d}{d\tau}\sum_{i=1}^2\Big\{\frac{A_i}{\beta_i}\vp_{i1}+\chi_iK_{i1}^{(1)}\Big\}\delta'(x-x^*)
+O_{\mathcal{D}'}(\varepsilon^2),\label{25a}
\end{align}
where $\dpsi=\beta_1(V_2-V_1)$.

On the contrary, to find $\pa (u^2)/\pa t$ with the same accuracy we should add to the leading term (\ref{26}) the next correction:
\begin{align}
&-\ve^2 a_2\Bigg\{
\sum_{i=1}^2\frac{\chi_i}{\beta_i}(2A_iS_i+S_i^2)
\notag\\
&+2\frac{G_1G_2}{\beta_2}\Big(\chi_2\lambda_{2,1}(\sigma_{12})+\frac{1}{\beta_2}\lambda_{2,1}^{(1)}(\sigma_{12})\Big)\Bigg\}\delta'(x-x^*)
+O_{\mathcal{D}'}(\varepsilon^3)\label{26a}.
\end{align}
Now we obtain
\begin{align}
\frac{\pa u^2}{\pa t}&=a_2\frac{d}{d\tau}\Big\{2\frac{\dpsi}{\beta_2}G_1G_2\lambda_{2,1}(\sigma_{12})
+\sum_{i=1}^2\frac{\dot{\psi}_0}{\beta_i}\big(2A_iS_i+S_i^2\big)\Big\}\delta(x-x^*)\notag\\
&
-\ve a_2\frac{d}{d\tau}\Bigg\{\dpsi\sum_{i=1}^2\frac{A_i^2}{\beta_i}\vp_{i1}+2\frac{\dpsi}{\beta_2}G_1G_2\big(\chi_2\lambda_{2,1}(\sigma_{12})
+\frac1{\beta_2}\lambda_{2,1}^{(1)}(\sigma_{12})\big)\label{26b}\\
&
+\sum_{i=1}^2\frac{\dpsi}{\beta_i}\chi_i\big(2A_iS_i+S_i^2\big)\Bigg\}\delta'(x-x^*)
-\ve a_2\sum_{i=1}^2V_i\frac{A_i^2}{\beta_i}\delta'(x-\vp_{i})+O_{\mathcal{D}'}(\varepsilon^2).\notag
\end{align}
Calculating weak expansions for the other terms from the left-hand
sides of (\ref{19}), (\ref{20}) and substituting them into
(\ref{19}), (\ref{20}) we obtain linear combinations of
$\delta'(x-\vp_{i})$, $i=1,2$, $\delta(x-x^*)$, and
$\delta'(x-x^*)$. Therefore, we pass to the following system of 8
equations for 8 unknowns:
\begin{align}
&P_{i,j}(A_i,\beta_i,V_i)=0, \quad j=1,2,\quad i=1,2,\label{28}\\
&\frac{d}{d\tau}Q_j(S_1,S_2,\sigma_{12})=0,\quad j=1,2,\label{29}\\
&\frac{d\vp_{j1}}{d\tau}=R_j(S_1,S_2,\sigma_{12}),\quad
j=1,2.\label{30}
\end{align}
The first four algebraic equations (\ref{28}) imply again the  relations
(\ref{3}) between $A_i$, $\beta_i$, and $V_i$. Furthermore,
functional equations (\ref{29}) allow to define $S_i$ with the
property (\ref{23}), whereas an analysis of the ODE (\ref{30})
justifies the existence of the required phase corrections $\vp_{i1}$
with the property (\ref{24}). By analogy with  (\ref{11}), (\ref{11a})
we call (\ref{28})-(\ref{30}) the Hugoniot-type conditions again.
\begin{figure}[ht]
\centering
\includegraphics[height=2in,width=3.5in]{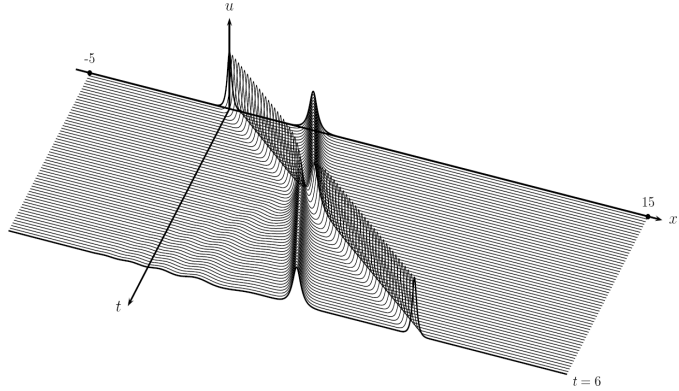}
\caption{Evolution of two solitary waves  for $\ve=0.1$}
\label{Fig1}
\end{figure}
Results of numerical simulations \cite{GarOm2, GarOm3} confirm the traced asymptotic
analysis (see Figure 1).

Finally let us note that the two-wave problem for the equation (\ref{1}) has been considered recently in the framework of another approach
 but for small  amplitudes   \cite{Merle2}.

\subsection{The next problem: three wave interaction}
Further numerical investigation of the GKdV-4 equation showed  that $N$ solitary waves collide
elastically (in the leading term) at least for $N\leq 5$, see \cite{GarOm2, GarOm3} and Figures 2--4.
\begin{figure}[ht]
\centering
\includegraphics[height=2in,width=3.5in]{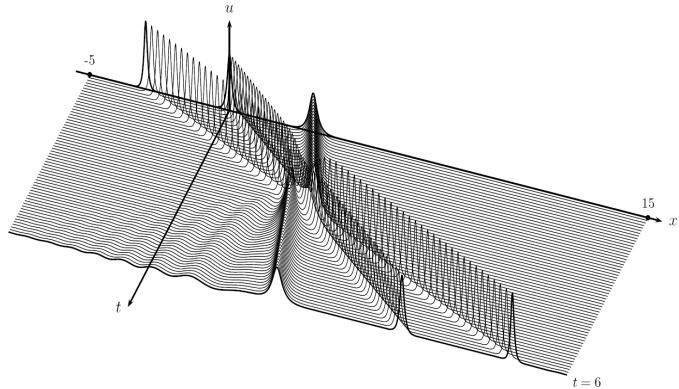}
\caption{Evolution of the soliton triplet for $\ve=0.1$}
\label{Fig2}
\end{figure}
Thus, there appears the problem of describing the interaction
of three (and more) solitary waves. It seemed that it was enough to
repeat the same procedure as above, now for three waves. However,
it is easy to recognize that the corresponding 4 equations
(\ref{29}), (\ref{30}) will contain now 6 free functions.
Obviously, this can not be any adequate description of the
solution.

 Therefore, we should transform the conception of the asymptotic
solution. To do it let us recall    two first
conservation laws for the equation (\ref{1}) (in  the differential form):
\begin{equation}\label{31}
\frac{\pa u}{\pa t}+\frac{\pa }{\pa x}\Big\{u^4+\varepsilon^2
\frac{\partial^2u}{\pa x^2}\Big\}= 0,
\end{equation}
\begin{equation}\label{32}
\frac{\pa u^2}{\pa t}+\frac{\pa }{\pa
x}\Big\{\frac85u^5-3\varepsilon^2 \Big(\frac{\partial u}{\pa
x}\Big)^2+\varepsilon^2 \frac{\partial^2u^2}{\pa x^2}\Big\}= 0.
\end{equation}
Comparing the left-hand sides of (\ref{31}), (\ref{32}) with
(\ref{19}), (\ref{20}) we conclude that Definition 1 calls a
function to be a ''weak asymptotic solution" if it satisfies the
conservation laws (\ref{31}), (\ref{32}) in the sense $
O_{\mathcal{D}'}(\varepsilon^2)$.
\begin{figure}[ht]
\centering
\includegraphics[height=2in,width=3.5in]{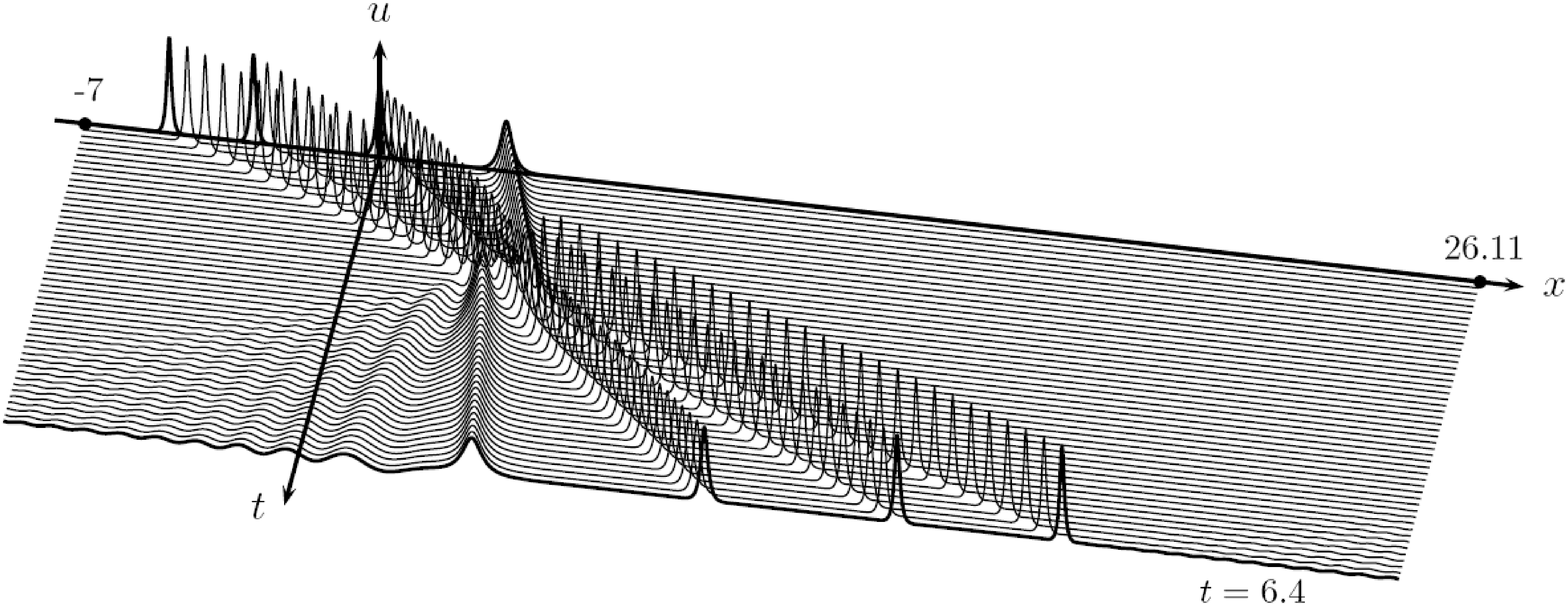}
\caption{Evolution of 4 solitons for $\ve=0.1$}
\end{figure}
Next note that perturbations of (\ref{1}) imply corresponding
perturbations of the conservation laws. For (\ref{5d}), instead of
(\ref{31}), (\ref{32}), we have the relations
\begin{equation}\label{33}
\frac{\pa u}{\pa t}+\frac{\pa }{\pa x}\Big\{u^4+\varepsilon^2
\frac{\partial^2u}{\pa x^2}\Big\}=R,
\end{equation}
\begin{equation}\label{34}
\frac{\pa u^2}{\pa t}+\frac{\pa }{\pa
x}\Big\{\frac85u^5-3\varepsilon^2 \Big(\frac{\partial u}{\pa
x}\Big)^2+\varepsilon^2 \frac{\partial^2u^2}{\pa x^2}\Big\}=
2R u.
\end{equation}

Reverting to the single-phase asymptotic solution (\ref{5e}) one
can easily establish that the orthogonality condition (\ref{8}), and therefore the equation (\ref{11a}),
 is the integral form of
(\ref{34}), calculated for (\ref{5e}) with the accuracy
$O_{\mathcal{D}'}(\varepsilon^2)$. At the same time (\ref{9}), and thus the equality (\ref{11b}),
 is
the integral form of (\ref{33}) $\mod
O_{\mathcal{D}'}(\varepsilon^2)$, calculated for (\ref{5e}), where
$Y_1$ has the form (\ref{14a}) and $u_1^-$ satisfies the equation
(\ref{15}).
\begin{figure}[ht]
\centering
\includegraphics[height=2in,width=5in]{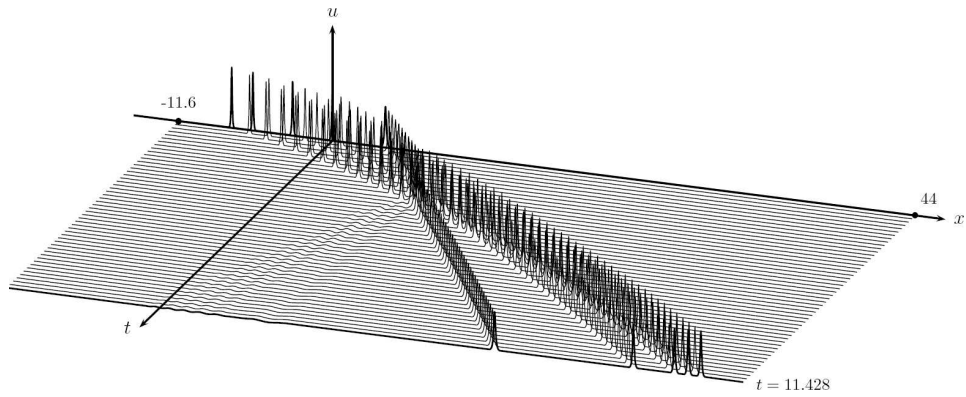}
\caption{Evolution of 5 solitons for $\ve=0.1$}
\end{figure}
Note also that the second "conservation law"
(\ref{34}) has been used in the one-phase situation  to define the principal term of the
asymptotic solution, whereas  (\ref{33}) has been used only to
define  the first correction $Y_1$.

Therefore, we see that to define the principal asymptotics term there has been
used only one conservation law for the single-phase solution, and
two conservation laws for the two-phase solution. So it is
natural to assume that to construct a three-phase asymptotics we
should add to (\ref{31}), (\ref{32}) the third conservation law,
namely
\begin{equation}\label{35}
\frac{\pa }{\pa t}\Big\{\Big(\ve\frac{\pa u}{\pa
x}\Big)^2-\frac25u^5\Big\}-\frac{\pa }{\pa x}\Big\{2\ve^2\frac{\pa
u}{\pa t}\frac{\pa u}{\pa x}+\Big(u^4+\varepsilon^2
\frac{\partial^2 u}{\pa x^2}\Big)^2\Big\}= 0.
\end{equation}

\section{Asymptotic construction}
Let us consider the equation (\ref{1}) with the Cauchy data
\begin{equation}\label{36}
u|_{t=0}=\sum_{i=1}^3A_i\om\left(\beta_i
\frac{x-x_{i0}}{\varepsilon}\right),
\end{equation}
where $A_1<A_2<A_3$, $x_{10}>>x_{20}>>x_{30}$.

The arguments considered in the previous subsection imply the following
\begin{defi}\label{def2.3}
A sequence $u(t, x, \varepsilon )$, belonging to
$\mathcal{C}^\infty (0, T; \mathcal{C}^\infty (\mathbb{R}_x^1))$
for $\varepsilon =\const> 0$ and belonging to $\mathcal{C} (0, T;
\mathcal{D}' (\mathbb{R}_x^1))$ uniformly in $\varepsilon$, is
called the weak asymptotic mod $ O_{\mathcal{D}'}(\varepsilon^2)$
solution of the problem (\ref{1}), (\ref{36}) if the relations
(\ref{19}), (\ref{20}), and
\begin{align}
&\frac{\pa }{\pa t}\Big\{\int_{-\infty}^\infty\Big(\ve\frac{\pa
u}{\pa x}\Big)^2\psi dx-\frac25\int_{-\infty}^\infty u^5\psi dx
\Big\}\label{37}\\
&+\int_{-\infty}^\infty\Big\{2\ve^2\frac{\pa u}{\pa t}\frac{\pa
u}{\pa x}+u^8+\Big(\varepsilon^2 \frac{\partial^2 u}{\pa
x^2}\Big)^2+2u^4\varepsilon^2 \frac{\partial^2 u}{\pa
x^2}\Big\}\frac{\partial \psi}{\pa x}dx= O
(\varepsilon^2)\notag
\end{align}
hold uniformly in $t$ for any test function $\psi = \psi(x) \in
\mathcal{D} (\mathbb{R}^1)$.
\end{defi}

To construct the asymptotic solution we write the ansatz in the
form similar to (\ref{22}), namely
\begin{equation}\label{38}
u=\sum_{i=1}^3G_i\om\left(\beta_i
\frac{x-\vp_i}{\varepsilon}\right),
\end{equation}
where the same notation and hypothesis (\ref{22a})-(\ref{24}) are assumed with obvious corrections:
the "fast time" is defined now using the distance between the first and third trajectories,
\be\tau=\beta_1\big(\vp_{30}(t)-\vp_{10}(t)\big)/\ve,\label{38a}\ee
and we  suppose the intersection of  all trajectories $x=\vp_{i0}(t)$, $i=1,2,3$, at
the same point $(x^*,t^*)$.

The technic of our approach has been explained in Subsection 1.3. So we clarify here some new detail only using as an example $u^m$.
All others explicit formulas for the asymptotic expansions are presented  in the Appendix.
\begin{lem} \label{lem1}
Let the assumptions (\ref{23}), (\ref{24}) for $u$ of the form
(\ref{38})  be satisfied. Then the following asymptotic expansions hold:
\begin{align}
&u^m=\ve a_m    \Big\{\sum_{i=1}^3K_{i0}^{(m)}\delta(x-\vp_{i})+\mathfrak{R}_m\delta(x-x^*)\Big\}
+O_{\mathcal{D}'}(\varepsilon^2),\label{39}\\
&\frac{\pa u^m}{\pa t}= a_m \dpsi\frac{d\mathfrak{R}_m}{d\tau}\delta(x-x^*)-\ve a_m\sum_{i=1}^3V_iK_{i0}^{(m)}\delta'(x-\vp_{i}) \notag\\
&-\ve a_m\dpsi\frac{d}{d\tau}\Big\{\sum_{i=1}^3K_{i0}^{(m)}\vp_{i1}+\mathfrak{R}_m^{(1)}\Big\}\delta'(x-x^*)
+O_{\mathcal{D}'}(\varepsilon^2),\label{39a}
\end{align}
where $m\geq1$,
\begin{align}
&\dpsi=\beta_1(V_3-V_1),\; K_{i}^{(m)}=\frac{G_i^m}{\beta_i}, \; K_{i0}^{(m)}=\frac{A_i^m}{\beta_i},\; K_{i1}^{(m)}=K_{i}^{(m)}-K_{i0}^{(m)},\label{40}\\
&\mathfrak{R}_m=\sum_{i=1}^3K_{i1}^{(m)}+\sum_{l,n}R_{m,ln}+R_{m,123},\label{41}\\
&\mathfrak{R}_m^{(1)}=\sum_{i=1}^3\chi_iK_{i1}^{(m)}+\sum_{l,n}(\chi_nR_{m,ln}+C_{m,ln})+\chi_3R_{m,123}+C_{m,123},\label{42}\\
&R_{m,ln}=\beta_l\sum_{k=1}^{m-1}C_m^kK_l^{(m-k)}K_n^{(k)}\lambda_{m,k}(\sigma_{ln}),\label{43}\\
&R_{m,123}=\beta_1\beta_2\sum_{j=2}^{m-1}\sum_{k=1}^{j-1}C_m^jC_j^kK_1^{(m-j)}K_2^{(j-k)}K_3^{(k)}\lambda_{m,123}^{(0),(j,k)},\label{44}\\
&C_{m,ln}=\sum_{k=1}^{m-1}C_m^k\theta_{ln}K_l^{(m-k)}K_n^{(k)}\lambda_{m,k}^{(1)}(\sigma_{ln}),\label{45}\\
&C_{m,123}=\beta_1\theta_{23}\sum_{j=2}^{m-1}\sum_{k=1}^{j-1}C_m^jC_j^kK_1^{(m-j)}K_2^{(j-k)}K_3^{(k)}\lambda_{m,123}^{(1),(j,k)},\label{46}
\end{align}
$C_m^k$ are the binomial coefficients,  the notation (\ref{27}) has been used,  and
\be
\label{47}
\sum_{l,n}f_{ln}\ipd f_{12}+f_{13}+f_{23},\quad \sigma_{ln}=\beta_l\frac{\vp_l-\vp_n}{\ve},\quad \chi_i=\frac{V_i}{\dpsi}\tau+\vp_{i1}.
\ee
Furthermore,
\be
\lambda_{m,123}^{(i),(j,k)}=\frac{1}{a_m}\int_{-\infty}^\infty\eta^i
\om^{m-j}(\eta_{13})\om^{j-k}(\eta_{23})\om^{k}(\eta)d\eta,\quad i=0 \quad\text{or}\quad i=1.\label{51}
\ee
\end{lem}
To prove the lemma
let us separate all the terms of $u^m$ into three groups: one-phase, two-phase and three-phase functions:
\be\label{53}
 u^m=\sum_{i=1}^{3}Y_i^m+\sum_{l,n}\sum_{k=1}^{m-1}C_m^kY_l^{m-k}Y_n^k +\sum_{j=1}^{m-1}\sum_{k=1}^{j-1}C_m^jC_j^kY_1^{m-j}Y_2^{j-k}Y_3^k,
\ee
where $Y_i=G_i\om\big(\beta_i(x-\vp_i)/\varepsilon\big)$.
Now considering $u^m$ in the weak sense we change the variable: $x=\vp_i+\ve\eta/\beta_i$,
$x=\vp_n+\ve\eta/\beta_n$, and $x=\vp_3+\ve\eta/\beta_3$ respectively for the integrals of the groups.
Next, preparing the same transformations as in Subsection 1.3 and applying  (\ref{23e}) again we pass to the formula (\ref{39}).

In the same manner one can prove the following proposition:
\begin{lem} \label{lem2}
Let the assumptions (\ref{23}), (\ref{24}) be satisfied for $u$ of the form
(\ref{38}). Then the following asymptotic expansions hold:
\begin{align}
&(\ve u_x)^2=\ve a'_2    \Big\{\sum_{i=1}^3\beta_i^2K_{i0}^{(2)}\delta(x-\vp_{i})+\mathfrak{R}_{(1),2}\delta(x-x^*)\Big\}
+O_{\mathcal{D}'}(\varepsilon^2),\label{54}\\
&\frac{\pa }{\pa t}(\ve u_x)^2= a'_2 \dpsi\frac{d\mathfrak{R}_{(1),2}}{d\tau}\delta(x-x^*)
-\ve a'_2\sum_{i=1}^3\beta_i^2V_iK_{i0}^{(2)}\delta'(x-\vp_{i}) \notag\\
&-\ve a'_2\dpsi\frac{d}{d\tau}\Big\{\sum_{i=1}^3\beta_i^2K_{i0}^{(2)}\vp_{i1}+\mathfrak{R}_{(1),2}^{(1)}\Big\}\delta'(x-x^*)
+O_{\mathcal{D}'}(\varepsilon^2),\label{55}\\
&(\ve^2 u_{xx})^2=\ve a''_2    \Big\{\sum_{i=1}^3\beta_i^4K_{i0}^{(2)}\delta(x-\vp_{i})+\mathfrak{R}_{(2),2}\delta(x-x^*)\Big\}
+O_{\mathcal{D}'}(\varepsilon^2),\label{56}\\
&\ve^2u^4 u_{xx}=\ve a_{23} \Big\{-4 \sum_{i=1}^3\beta_i^2K_{i0}^{(5)}\delta(x-\vp_{i})+\mathfrak{L}\delta(x-x^*)\Big\}
+O_{\mathcal{D}'}(\varepsilon^2),\label{57}\\
&\ve^2 u_xu_t=-\ve a'_2    \sum_{i=1}^3\beta_i^2V_iK_{i0}^{(2)}\delta(x-\vp_{i})-\ve a'_2  \mathfrak{P}\delta(x-x^*)
+O_{\mathcal{D}'}(\varepsilon^2),\label{58}
\end{align}
where $a'_2$, $\dpsi$,  and $K_{i0}^{(2)}$ are defined in (\ref{3a}), (\ref{40}), $\mathfrak{P}=\dpsi(\mathfrak{S}+\mathfrak{S}_G)+\mathfrak{M}$, $\mathfrak{S}$, $\mathfrak{S}_G$, $\mathfrak{M}$, and other notation are deciphered in Attachment, Subsections 6.1 and 6.2.
\end{lem}
Now we substitute the expansions  (\ref{39}), (\ref{39a}), and (\ref{54})-(\ref{58})
into (\ref{19}), (\ref{20}), (\ref{37}) and obtain the similar (\ref{28})-(\ref{30}) system. Namely, the algebraic system for each $i=1,2,3$:
\begin{align}
&-V_iK_{i0}^{(1)}+a_4K_{i0}^{(4)}=0, \label{59}\\
&-a_2V_iK_{i0}^{(2)}+\frac85a_5K_{i0}^{(5)}-3a'_2\beta_i^2K_{i0}^{(2)}=0, \label{60}\\
&-a'_2\beta_i^2V_iK_{i0}^{(2)}+\frac25a_5V_iK_{i0}^{(5)}+2a'_2\beta_i^2V_iK_{i0}^{(2)}\notag\\
&\qquad\qquad-a_8K_{i0}^{(8)}+8a_{23}\beta_i^2K_{i0}^{(5)}-a''_2\beta_i^4K_{i0}^{(2)}=0, \label{61}
\end{align}
the system of functional equations:
\begin{align}
&\sum_{i=1}^3K_{i1}^{(1)}=0,\label{62}\\
&\mathfrak{R}_2=0,\label{63}\\
&a'_2\mathfrak{R}_{(1),2}-\frac25a_5\mathfrak{R}_{5}=0, \label{64}
\end{align}
and the system of ordinary differential equations:
\begin{align}
&-\dpsi\frac{d}{d\tau}\Big\{\sum_{i=1}^3K_{i0}^{(1)}\vp_{i1}+\chi_iK_{i1}^{(1)}\Big\}+a_4\mathfrak{R}_4=0,\label{65}\\
&-a_2\dpsi\frac{d}{d\tau}\Big\{\sum_{i=1}^3K_{i0}^{(2)}\vp_{i1}+\mathfrak{R}_{2}^{(1)}\Big\}+\frac85a_5\mathfrak{R}_5
-3a'_2\mathfrak{R}_{(1),2}=0,\label{66}\\
&\dpsi\frac{d}{d\tau}\Big\{-a'_2\Big(\sum_{i=1}^3\beta_i^2K_{i0}^{(2)}\vp_{i1}+\mathfrak{R}_{(1),2}^{(1)}\Big)
+\frac25a_5\Big(\sum_{i=1}^3K_{i0}^{(5)}\vp_{i1}+\mathfrak{R}_{5}^{(1)}\Big)\Big\}\notag\\
&\qquad\qquad+2a'_2\mathfrak{P}-a_8\mathfrak{R}_{8}-2a_{23}\mathfrak{L}-a''_2\,\mathfrak{R}_{(2),2}=0.\label{67}
\end{align}
Let us overcome the first obstacle: for each $i$ the system (\ref{59})-(\ref{61}) of three equation contains only two free parameters $A_i$,
$V_i$.
\begin{lem} \label{lem3}
Let  $\om(\eta)$, $A_i=A(\beta_i)$, and
$V_i=V(\beta_i)$ be of the form
(\ref{2})-(\ref{3b}). Then the equalities (\ref{59})-(\ref{61}) are satisfied uniformly in $\beta_i>0$.
\end{lem}
\begin{proof}
Obviously, equations (\ref{59}),  (\ref{60}) coincide with (\ref{28}) and
imply again the formulas  (\ref{2})-(\ref{3b}). Substituting them into (\ref{61}), we transform it to the following form:
\begin{equation}\label{68}
a_2a_4^2-a_8-a''_2\gamma^6+8a_{23}\gamma^3=0.
\end{equation}
Next we note that $\om(\eta)$ satisfies the model equation
\begin{equation}\label{69}
\gamma^3\frac{d^2\omega}{d\eta^2}=a_4\omega-\omega^4.
\end{equation}
Multiplying (\ref{69}) for $\om''$ and integrating, we obtain the identity
\begin{equation}\label{70}
4a_{23}=\gamma^3a''_2+a_4a'_2.
\end{equation}
On the other hand, integrating the squares of the left-hand and right-hand parts of (\ref{69}),
we pass to another identity:
\begin{equation}\label{71}
a_{8}=\gamma^6a''_2-a_2a_4^2+2a_4a_5.
\end{equation}
This and (\ref{3b}) verify the equality (\ref{68}).
\end{proof}
Since the system of six equations (\ref{62})-(\ref{67}) contains six free functions,
we obtain the first formal result
\begin{teo} \label{teo1}
Let the system  (\ref{62})-(\ref{67}) have a solution which satisfies the assumptions of
the form (\ref{23}), (\ref{24}). Then the solitary waves (\ref{38}) collide preserving   $\mod O_{\mathcal{D}'}(\varepsilon^2)$
the KdV-type scenario of interaction.
\end{teo}
Moreover, similar to the Rankine-Hugoniot condition, which is simply the conservation law for the shock-wave solution,
the Hugoniot-type conditions (\ref{59})-(\ref{67}) imply the verification of some conservation laws:
\begin{teo} \label{teo2}
Let the assumptions of Theorem 1 be satisfied. Then  the ansatz (\ref{38}) is a
$\mod O_{\mathcal{D}'}(\varepsilon^2)$  asymptotic solution of the equation (\ref{1}) if and only if (\ref{38})  satisfies the
conservation laws
\be\label{72}
\frac{d}{dt}\int_{-\infty}^\infty u dx=0,\; \frac{d}{dt}\int_{-\infty}^\infty u^2 dx=0,\;
\frac{d}{dt}\int_{-\infty}^\infty \Big\{\Big(\ve\frac{\pa u}{\pa
x}\Big)^2-\frac25u^5\Big\} dx=0,
\ee
and the energy relations
\begin{align}
&\frac{d}{dt}\int_{-\infty}^\infty xudx-\int_{-\infty}^\infty u^4dx=0,\notag\\
&\frac{d}{dt}\int_{-\infty}^\infty xu^2dx-\frac85\int_{-\infty}^\infty u^5dx+3\int_{-\infty}^\infty\left(\ve\frac{\pa u}{\pa x}\right)^2=0,
\label{73}\\
&\frac{d}{dt}\left\{\int_{-\infty}^\infty x\Big(\ve\frac{\pa u}{\pa x}\Big)^2dx-\frac25\int_{-\infty}^\infty xu^5dx\right\}\notag\\
&\qquad\qquad+2\ve^2\int_{-\infty}^\infty\frac{\pa u}{\pa t}\frac{\pa u}{\pa x}dx
+\int_{-\infty}^\infty\left(u^4+\ve^2\frac{\pa^2 u}{\pa x^2}\right)^2dx=0.
\notag
\end{align}
\end{teo}
To prove this conclusion it is enough to rewrite the equations  (\ref{59})-(\ref{67}) in the integral form.
\section{Analysis of the Hugoniot-type conditions}
\subsection{Transformations}
Let us pay the attention to the equations (\ref{62})-(\ref{64}). Normalization
\be
\kappa_i=\gamma\beta_3^{1/3}S_i/\beta_i\label{74}
\ee
implies
\be
K^{(m)}_i=\frac{\beta_3^{2m/3-1}}{\gamma^m}\,\theta_{i3}^{2m/3-1}\,\Lambda^{\,m}_i,\quad{\text{where}}\quad\Lambda_i=1+\theta_{i3}^{1/3}\kappa_i.\label{75}
\ee
We denote $\ol{\mathfrak{R}}_m=\gamma^m\mathfrak{R}_m/\beta_3^{2m/3-1}$ and obtain:
\begin{align}
&\ol{\mathfrak{R}}_m=\sum_{i=1}^3\theta_{i3}^{2m/3-1}(\Lambda^{\,m}_i-1)+\sum_{l,n}\ol{R}_{m,ln}+\ol{R}_{m,123}\label{77}\\
&\ol{R}_{m,ln}=\sum_{k=1}^{m-1}C_m^k\theta_{l3}^{2(m-k)/3}\theta_{n3}^{2k/3-1}\,\Lambda^{\,m-k}_l\,\Lambda^{\,k}_n\lambda_{m,k}(\sigma_{ln}),\label{78}\\
&\ol{R}_{m,123}=\sum_{j=2}^{m-1}\sum_{k=1}^{j-1}C_m^jC_j^k\theta_{13}^{2(m-j)/3}\theta_{23}^{2(j-k)/3}\,\Lambda^{\,m-j}_1\,\Lambda^{\,j-k}_2
\,\Lambda^{\,k}_3\lambda_{m,123}^{(0),(j,k)}.\label{79}
\end{align}
This and similar formulas for $\mathfrak{R}_{(1),2}\ipd\beta_3^{7/3}\ol{\mathfrak{R}}_{(1),2}/\gamma^2$ (see Attachment) allow us to transform (\ref{62})-(\ref{64}) to the following form:
\begin{align}
&\sum_{i=1}^3\kappa_i=0,\label{80}\\
&\sum_{i=1}^3\theta_{i3}^{1/3}(\Lambda^{\,2}_i-1)+2\sum_{l,n}\theta_{l3}^{1/3}\theta_{ln}^{1/3}\,\Lambda_l\,\Lambda_n\,\lambda_{2,1}(\sigma_{ln})=0,\label{81}\\
&\sum_{i=1}^3\theta_{i3}^{7/3}(\Lambda^{\,2}_i-1)+2\sum_{l,n}
\theta_{l3}^{5/3}\theta_{n3}^{2/3}\,\Lambda_l\,\Lambda_n\lambda_{I1}^{(0)}(\sigma_{ln})-\frac43\ol{\mathfrak{R}}_5=0, \label{82}
\end{align}
where the equalities (\ref{3b}), the notation (\ref{77})-(\ref{79}), and (\ref{3A}) have been used.

Next let us simplify the equations (\ref{65})-(\ref{67}). We note firstly that
in view of (\ref{62}) and the identity
\be
\beta_l(\chi_l-\chi_n)=\sigma_{ln} \label{83}
\ee
it is possible to eliminate $\chi_i$ from the left-hand side of (\ref{65}), since
$$
\sum_{i=1}^3K_{i0}^{(1)}\vp_{i1}+\chi_iK_{i1}^{(1)}=\sum_{i=1}^3K_{i0}^{(1)}\vp_{i1}+\frac{\sigma_{12}}{\beta_1}K_{11}^{(1)}
-\frac{\sigma_{23}}{\beta_2}K_{31}^{(1)}.
$$
In the same manner, applying  (\ref{63}) and (\ref{64}), we simplify the equations (\ref{66}), (\ref{67}).
Thus, we transform (\ref{65})-(\ref{67}) to the following form:
\begin{align}
&\dpsi\frac{d}{d\tau}\Big\{\sum_{i=1}^3K_{i0}^{(1)}\vp_{i1}+\frac{\sigma_{12}}{\beta_1}K_{11}^{(1)}
-\frac{\sigma_{23}}{\beta_2}K_{31}^{(1)}\Big\}=f,\label{84}\\
&\dpsi\frac{d}{d\tau}\Big\{\sum_{i=1}^3K_{i0}^{(2)}\vp_{i1}+\sum_{l,n}C_{2,ln}+\frac{\sigma_{12}}{\beta_1}K_{11}^{(2)}\notag\\
&\qquad\qquad\qquad\qquad-\frac{\sigma_{23}}{\beta_2}\big(K_{31}^{(2)}+R_{2,13}+R_{2,23}\big)\Big\}=F,\label{85}\\
&\dpsi\frac{d}{d\tau}\Big\{\sum_{i=1}^3\Big(\beta_i^2K_{i0}^{(2)}-\frac43\gamma^3K_{i0}^{(5)}\Big)\vp_{i1}+\mathfrak{K}\Big\}
-2\mathfrak{S}=\mathfrak{F}, \label{86}
\end{align}
where
\be
f=a_4\mathfrak{R}_4,\quad F=\frac{a'_2}{a_2}\mathfrak{R}_{(1),2},\quad \mathfrak{F}=2\mathfrak{M}-\frac{a_8}{a'_2}\mathfrak{R}_{8}-2\frac{a_{23}}{a'_2}\mathfrak{L}-\frac{a''_2}{a'_2}\mathfrak{R}_{(2),2},
\label{87}
\ee
and the function  $\mathfrak{K}$ is described in Attachment (see formula (\ref{5A})).

The second step is the elimination of $\vp_{i1}$ from the model system. To do it we divide $\sigma_{ln}$ into the growing ($\os$) and the bounded ($\ts$) parts:
\be\label{89}
\sigma_{ln}=\os+\ts,\quad\os\ipd\frac{\beta_l}{\dpsi}(V_l-V_n)\tau
\ee
and rewrite the identity (\ref{83}):
\be\label{90}
\ts=\beta_l(\vp_{l1}-\vp_{n1}).
\ee
Thus
\be\label{91}
\vp_{11}=\frac{\tilde{\sigma}_{12}}{\beta_1}+\vp_{21},\quad \vp_{31}=-\frac{\tilde{\sigma}_{23}}{\beta_2}+\vp_{21}.
\ee
Substituting (\ref{91}) into (\ref{84}) we obtain
\be
\dpsi\frac{d\vp_{21}}{d\tau}=-\frac{\dpsi}{r_1}\frac{d}{d\tau}\Big\{\frac{\tilde{\sigma}_{12}}{\beta_1}K_{10}^{(1)}
+\frac{\sigma_{12}}{\beta_1}K_{11}^{(1)}-\frac{\tilde{\sigma}_{23}}{\beta_2}K_{30}^{(1)}
-\frac{\sigma_{23}}{\beta_2}K_{31}^{(1)}\Big\}+\frac {f}{r_1}.\label{92}
\ee
Here and in what follows we use the notation
\be
r_j=\sum_{i=1}^3K_{i0}^{(j)}\quad \text{for}\quad  j=1\quad \text{and}\quad j=2.\label{92a}
\ee
Next we use the equalities (\ref{91}), (\ref{92}),  and
\be \label{93}
\tilde{\sigma}_{13}=\tilde{\sigma}_{12}+\theta_{12}\tilde{\sigma}_{23},\quad \ol{\sigma}_{13}=\ol{\sigma}_{12}+\theta_{12}\ol{\sigma}_{23},
\ee
and  rewrite  (\ref{85}), (\ref{86}) as  equations for new unknowns $\tilde{\sigma}_{12}$, $\tilde{\sigma}_{23}$. After normalization (\ref{74}) we pass to the following model equations:
\begin{align}
&\dpsi\frac{d}{d\tau}\Big\{p_{10}\frac{\tilde{\sigma}_{12}}{\beta_1}+p_{11}\frac{{\sigma}_{12}}{\beta_1}
-p_{30}\frac{\tilde{\sigma}_{23}}{\beta_2}
-p_{31}\frac{{\sigma}_{23}}{\beta_2}+\sum_{l,n}C_{2,ln}\Big\}
=F-\frac{r_2}{r_1}f,\label{94}\\
&\dpsi \Bigg\{\frac{d}{d\tau}\Big\{e_{10}\frac{\tilde{\sigma}_{12}}{\beta_1}+e_{11}\frac{\sigma_{12}}{\beta_1}\Big\}
+2\Psi\frac{d}{d\tau}\Big\{K_{11}^{(1)}\frac{\sigma_{12}}{\beta_1}\Big\}
-\frac{d}{d\tau}\Big\{e_{30}\frac{\tilde{\sigma}_{23}}{\beta_2}+e_{31}\frac{\sigma_{23}}{\beta_2}\Big\} \notag\\
&-2\Psi\frac{d}{d\tau}\Big\{K_{31}^{(1)}\frac{\sigma_{23}}{\beta_2}\Big\}
+ r_{1}\frac{d\mathfrak{K}_C}{d\tau}-2r_{1}\mathfrak{S}_G\Bigg\}
=\mathfrak{F}_1,\label{95}
\end{align}
where $\mathfrak{F}_1=r_1\mathfrak{F}+(2\Psi+\sum_{i=1}^3q_{i0}^{(2)})f$ and the coefficients $p^{(k)}_i$, $e_{ki}$, $q_{ik}^{(m)}$, $\Psi$
are presented in Attachment (see formulas (\ref{6A}) - (\ref{11A})).

\subsection{Asymptotic analysis}

To simplify the further analysis let us assume that
\begin{equation}\label{96}
\theta_{23}^{1/3}=\mu, \;
\theta_{12}^{2/3}=\mu^{1+\alpha},\;\text{where}\;\alpha\in[0,1)\;\text{and}\;\mu\;\text{is
sufficiently small}.
\end{equation}
We look for the asymptotic solution of the system (\ref{80})-(\ref{82}) in the form:
\be\label{97}
\kappa_1=\frac12\mu^\alpha(y_1-\mu^{2-\alpha}x_0),\quad \kappa_2=-\frac12\mu^\alpha(y_1+\mu^{2-\alpha}x_0),\quad \kappa_3=\mu^2x_0,
\ee
where $x_0$ and $y_1$ are free functions. Then (\ref{80}) is satisfied, whereas (\ref{81}) and (\ref{82}) imply the system:
\begin{align}
&2x_0-\mu^{\alpha}y_1\Big\{1+\mu\lambda_{2,1}(\sigma_{23})-\mu^{1+\alpha}(1-\lambda_{2,1}(\sigma_{12})+\frac14y_1)-
\mu^{3(1+\alpha)/2}\lambda_{2,1}(\sigma_{12})\Big\}\notag\\
&\qquad=-2\lambda_{2,1}(\sigma_{23})-2\mu^\alpha \lambda_{2,1}(\sigma_{12})
-2\mu^{1+\alpha} \lambda_{2,1}(\sigma_{13})+O_\lambda(\mu^{2}),\label{98}\\
&\Big\{7+40\mu^2\lambda_{5,4}(\sigma_{23})+\frac{37}{2}\mu^{2}x_0\Big\}x_0-5\mu^{1+\alpha}y_1\lambda_{5,4}(\sigma_{23})\notag\\
&\qquad=-10\lambda_{5,4}(\sigma_{23})-10\mu^{1+\alpha}\lambda_{5,4}(\sigma_{13})+O_\lambda(\mu^{2}+\mu^{(3+7\alpha)/2}).\label{99}
\end{align}
Here and in what follows we denote
\be
f(\sigma,\mu)=O_\lambda(\mu^{k})\quad\text{if}\quad \max_{\sigma}|f(\sigma,\mu)|\leq c\mu^k \label{99a}
\ee
and $f(\sigma,\cdot)$ belongs to the Schwartz space.

It is easy to see that the compatibility of the equations (\ref{98}) and (\ref{99}) requires the condition:
$10\lambda_{5,4}(\sigma_{23})=7\lambda_{2,1}(\sigma_{23})+O_\lambda(\mu^{\alpha})$.
\begin{lem} \label{lem4}
Let  $\om(\eta)$ be of the form
(\ref{2}). Then
\be\label{101}
10\lambda_{5,4}(\sigma_{ln})=7\lambda_{2,1}(\sigma_{ln})+3\theta_{ln}\lambda_{I1}^{(0)}(\sigma_{ln})
\ee
for all indices $l,n$.
\end{lem}
To prove the lemma it is enough to
use again the equation (\ref{43}) and the identities
(\ref{3a}), (\ref{3b}), (\ref{42}).

Now we set
\be\label{102}
x_0=-\lambda_{2,1}(\sigma_{23})+\mu^\alpha x_1
\ee
and transform (\ref{98}), (\ref{99}) to the final form:
\begin{align}
&2x_1-r_{12}y_1=-2\lambda_{2,1}(\sigma_{12})-2\mu\lambda_{2,1}(\sigma_{13})+O_\lambda(\mu^{2-\alpha}),\label{103}\\
&r_{21}x_1-\frac57\mu \lambda_{5,4}(\sigma_{23})y_1=-\mu\lambda_{2,1}(\sigma_{13})-\frac37\mu^{3-\alpha}\lambda_{I1}^{(0)}(\sigma_{23})+O_\lambda(\mu^3),\label{104}
\end{align}
where
\begin{align}
&r_{12}=1+\mu\lambda_{2,1}(\sigma_{23})-\mu^{1+\alpha}\big(1-\lambda_{2,1}(\sigma_{12})+\frac14y_1\big)+O_\lambda(\mu^{3(1+\alpha)/2}),\notag\\
&r_{21}=1+\frac{19}{4}\mu^2\lambda_{2,1}(\sigma_{23})+O_\lambda(\mu^{2+\alpha}).\notag
\end{align}
Solving this system we obtain the asymptotic representation:
\begin{align}
&x_{1}=-\mu\big(\lambda_{2,1}(\sigma_{13})-\lambda_{2,1}(\sigma_{12})\lambda_{2,1}(\sigma_{23})\big)+O_\lambda(\mu^{2+\alpha}),\label{105}\\
&y_{1}=2\lambda_{2,1}(\sigma_{12})\Big(1+\mu^{1+\alpha}\big(1-\lambda_{2,1}(\sigma_{12})\big)\Big)+O_\lambda(\mu^{2-\alpha}).\label{106}
\end{align}
Combining (\ref{97}), (\ref{102}), (\ref{105}), and (\ref{106}) we conclude:
\begin{lem} \label{lem5}
Let  there exist functions $\vp_{i1}$, $i=1,2,3$, with the properties (\ref{24}) and let the condition (\ref{96}) be realized.
Then the system (\ref{80}) - (\ref{82}) has the unique solution
\begin{align}
&\kappa_{1}=\mu^\alpha\lambda_{2,1}(\sigma_{12})\big\{1+\mu^{1+\alpha}\big(1-\lambda_{2,1}(\sigma_{12})\big)\big\}
+O_\lambda(\mu^{2}),\label{107}\\
&\kappa_{2}=-\mu^\alpha\lambda_{2,1}(\sigma_{12})\big\{1+\mu^{1+\alpha}\big(1-\lambda_{2,1}(\sigma_{12})\big)\big\}
+O_\lambda(\mu^{2}),\label{108}\\
&\kappa_{3}=-\mu^2\lambda_{2,1}(\sigma_{23})+O_\lambda(\mu^{3+\alpha}),\label{109}
\end{align}
such that $S_i=\beta_i\kappa_i/\gamma\beta_3^{1/3}$  satisfy the assumptions  (\ref{23}).
\end{lem}

To complete the analysis we should prove the solvability of  the system (\ref{94}), (\ref{95}).
Taking into account  (\ref{74}), (\ref{96}), and (\ref{107}) - (\ref{108}), we obtain:
\begin{align}
&\tilde{E}_{11}\frac{d\tilde{\sigma}_{12}}{d\tau}-\theta_{12}\tilde{E}_{12}\frac{d\tilde{\sigma}_{23}}{d\tau}
=\tilde{F}_1,\label{110}\\
&\tilde{E}_{21}\frac{d\tilde{\sigma}_{12}}{d\tau}-\theta_{12}\tilde{E}_{22}\frac{d\tilde{\sigma}_{23}}{d\tau}
=\tilde{F}_2,\label{111}
\end{align}
where the coefficients $\tilde{E}_{ij}$ and right-hand sides $\tilde{F}_i$ are demonstrated in Attachment, Subsection 6.4.

It is easy to calculate that
\be
\det(\tilde{E}_{ij})=\theta_{12}\mu\Delta,\label{112}
\ee
where
\be
\Delta=\frac73+O(\mu^{(1+3\alpha)/2})+O_\lambda(\mu^{(3+\alpha)/2}+\mu^{2}).\notag
\ee
Thus, we transform the system (\ref{110}), (\ref{111}) to the standard form
\begin{equation}\label{113}
\frac{d\tilde{\sigma}_{12}}{d\tau}=\tilde{M}_{12}(\tau,\sigma_{12},\sigma_{23},\mu)/\Delta,\quad
\frac{d\tilde{\sigma}_{23}}{d\tau}=\tilde{M}_{23}(\tau,\sigma_{12},\sigma_{23},\mu)/\Delta,
\end{equation}
where
\begin{align}
&\tilde{M}_{12}=-\frac{20}{3}\mu z'(\sigma_{23})+O_\lambda(\mu^{(3+\alpha)/2}),
\quad z(\sigma)\ipd\sigma\lambda_{2,1}(\sigma),\label{112a}\\
&\tilde{M}_{23}=-2\mu^{-\alpha}\lambda_{2,1}(\sigma_{23})-\frac{7}{3}z'(\sigma_{12})
+O_\lambda(\mu^{(1-\alpha)/2}+\mu^{\alpha}),\label{112b}
\end{align}
and   the equalities (\ref{3b}), (\ref{70}), (\ref{71}), as well as the functional relation (\ref{101}) and
\be
a_8\lambda_{8,7}(\sigma_{ln})=a_2a_4^4\lambda_{2,1}(\sigma_{ln})-\gamma^3a_{23}\lambda_{4,3(2)}(\sigma_{ln})
+\theta_{ln}\gamma^3a_4a_2'\lambda_{I1}^{(0)}(\sigma_{ln})\label{114}
\ee
have been taken into account.

According to the notation (\ref{89}) and the first assumption of the
form (\ref{24}) we add to (\ref{113}) the "initial" condition:
\begin{equation}\label{114a}
\tilde{\sigma}_{12}\big|_{\tau\to-\infty}\to 0,\quad
\tilde{\sigma}_{23}\big|_{\tau\to-\infty}\to 0.
\end{equation}
Since $\tilde{M}_{ij}$ vanish with an exponential rate as $\tau\to\pm\infty$, it is easy to prove the solvability
of the problem (\ref{113}), (\ref{114a}). Next we note that $\lambda_{2,1}(\sigma_{ln})=\lambda_{2,1}(\dot{\ol{\sigma}}_{ln}\tau+\tilde{\sigma}_{ln})$. Since $\dot{\ol{\sigma}}_{23}=O(\mu^{-3(1+\alpha)/2})$ we find from  (\ref{113}), (\ref{112b}) that $\tilde{\sigma}_{23}(\tau)=O(\mu^{(3+\alpha)/2})$ for sufficiently large $\tau$, however it tends to the limiting value sufficiently slowly, with an exponent $O(\mu^{(3+\alpha)/2})$. Conversely, taking into account that $\dot{\ol{\sigma}}_{12}=O(\mu^{6})$, we obtain that $\tilde{\sigma}_{12}(\tau)=O(1)$ for sufficiently large $\tau$ and tends  to the limit    with an exponent $O(1)$.

The last step of the construction is the return to the phase corrections $\vp_{i1}$. In view of (\ref{91}), (\ref{92}) it is obvious that
the last assumption of the form (\ref{24}) is justified. This implies our main proposition
\begin{teo} \label{teo3}
Under the assumption   (\ref{96}) the asymptotic solution (\ref{38}) describes    $\mod O_{\mathcal{D}'}(\varepsilon^2)$ the KdV-type scenario of the solitary waves interaction.
\end{teo}
\section{Conclusion}
We looked for an approach to describe   solitary wave collisions avoiding  the use of explicit multi-soliton formulas. Surprisingly,
we came back to the ancient Whitham's idea to construct asymptotics with the help of conservation laws and a reasonable ansatz,
but in the framework of the weak asymptotics method. In our case  three conservation laws for  three waves have been utilized.
It is clear now how to generalize the approach: for $N$ waves $N$ conservation laws should  be used. On contrary, the existence
of  $N$ conservation laws does not imply the existence of  $N$-soliton type solution since some very astonishing  additional
conditions appear to guarantee both the  solvability of model equations (like (\ref{101})) and the regularity of the solutions (like (\ref{114})).
Furthermore, some questions remain  open, the first of them: how to choose the collection of conservation laws to describe $N$-soliton interaction and is it possible to change conservation laws to reasonable energy relations? At the same time we can formulate the main result of the paper: there is not a sharp frontier between integrable and nonintegrable equations: similar scenarios of the soliton interaction are realized, but with small corrections in the nonintegrable case.
\section{Acknowledgement}
The research was supported by  SEP-CONACYT under grant~178690
(Mexico).

\section{Attachment}
\subsection{Formulas for $(\ve u_x)^2$ and $(\ve^2 u_{xx})^2$}
\begin{align}
&\mathfrak{R}_{(k),2}=\sum_{i=1}^3\beta_i^{2k}K_{i1}^{(2)}+\sum_{l,n}R^{(k)}_{2,ln},\;R^{(k)}_{2,ln}=2\beta_l^{1+k}\beta_n^k K_l^{(1)}K_n^{(1)}\lambda_{I,k}^{(0)}(\sigma_{ln}),\label{1A}\\
&\mathfrak{R}_{(2),2}^{(1)}=\sum_{i=1}^3\beta_i^2\chi_iK_{i1}^{(2)}+\sum_{l,n}(\chi_nR^{(1)}_{2,ln}+C^{(1)}_{2,ln}),\label{2A}\\
&C^{(1)}_{2,ln}=2\beta_l^2K_l^{(1)}K_n^{(1)}\lambda_{I,1}^{(1)}(\sigma_{ln}),\;
\lambda_{I,1}^{(i)}(\sigma_{ln})=\frac{1}{a'_2}\int_{-\infty}^\infty\eta^i\om'(\eta_{ln})\om'(\eta)d\eta,\label{3A}\\
&\lambda_{I,2}(\sigma_{ln})=\frac{1}{a''_2}\int_{-\infty}^\infty\om''(\eta_{ln})\om''(\eta)d\eta,\quad
a''_2=\int_{-\infty}^\infty\Big(\om''(\eta)\Big)^2d\eta,\label{4A}
\end{align}
and $\om'(\eta)=d\om(\eta)/d\eta$, $\om''(\eta)=d^2\om(\eta)/d\eta^2$.
\subsection{Formulas for $\ve^2 u_xu_t$ and $\ve^2u^4 u_{xx}$}
\begin{align}
&\mathfrak{S}=\sum_{i=1}^3\beta_i^2K_{i}^{(2)}\frac{d\vp_{i1}}{d\tau}+\sum_{l,n}\beta_l^2\beta_n K_l^{(1)}K_n^{(1)}\Big(\frac{d\vp_{l1}}{d\tau}+\frac{d\vp_{n1}}{d\tau}\Big)\lambda_{I,1}^{(0)}(\sigma_{ln}),\notag\\
&\mathfrak{M}=\sum_{i=1}^3\beta_i^2V_iK_{i1}^{(2)}+\sum_{l,n}\mathfrak{M}_{ln},\,\mathfrak{M}_{ln}=\beta_l^2\beta_n K_l^{(1)}K_n^{(1)}\big(V_l+V_n\big)\lambda_{I,1}^{(0)}(\sigma_{ln}),\notag\\
&\mathfrak{S}_G=\sum_{l,n}\Big(G_l\frac{dG_n}{d\tau}-G_n\frac{dG_l}{d\tau}\Big)\lambda_{0I}(\sigma_{ln}),\;
\mathfrak{L}=-4\sum_{i=1}^3\beta_i^2K_{i1}^{(5)}+\sum_{l,n}\mathfrak{L}_{ln}+\mathfrak{Q},
\notag\\
&\mathfrak{L}_{ln}=\sum_{j=1}^4C^j_4\beta_l^3 K_l^{(5-j)}K_n^{(j)}\lambda_{4,j(1)}(\sigma_{ln})+\sum_{j=0}^3C^j_4\beta_l\beta_n^2 K_l^{(4-j)}K_n^{(j+1)}\lambda_{4,j(2)}(\sigma_{ln}),\notag\\
&\mathfrak{Q}=\sum_{j=1}^3C^j_4\Big\{\beta_1\beta_2\beta_3^2 K_1^{(4-j)}K_2^{(j)}K_3^{(1)}\lambda_{5,j(1)}\notag\\
&+\beta_1\beta_2^3 K_1^{(4-j)}K_2^{(1)}K_3^{(j)}\lambda_{5,j(2)}
+\beta_1^3\beta_2 K_1^{(1)}K_2^{(4-j)}K_3^{(j)}\lambda_{5,j(3)}\Big\}\notag\\
&+\sum_{j=2}^3\sum_{k=1}^{j-1}C^j_4C^k_j\beta_1\beta_2 K_1^{(4-j)}K_2^{(j-k)}K_3^{(k)}
\sum_{m=1}^{3}\beta_m^3K_m^{(1)}\lambda_{6,jkm},\notag\\
&\lambda_{0I}(\sigma_{ln})=\frac{1}{a'_2}\int_{-\infty}^\infty\om(\eta_{ln})\om'(\eta)d\eta,\quad
a_{23}=\int_{-\infty}^\infty\om^3(\eta)\Big(\om'(\eta)\Big)^2d\eta,\notag\\
&\lambda_{4,j(1)}(\sigma_{ln})=\frac{1}{a_{23}}\int_{-\infty}^\infty\om^{4-j}(\eta_{ln})\om^j(\eta)\om''(\eta_{ln})d\eta,\notag\\
&\lambda_{4,j(2)}(\sigma_{ln})=\frac{1}{a_{23}}\int_{-\infty}^\infty\om^{4-j}(\eta_{ln})\om^j(\eta)\om''(\eta)d\eta,\notag\\
&\lambda_{5,j(1)}=\frac{1}{a_{23}}\int_{-\infty}^\infty\om^{4-j}(\eta_{13})\om^j(\eta_{23})\om''(\eta)d\eta,\notag\\
&\lambda_{5,j(2)})=\frac{1}{a_{23}}\int_{-\infty}^\infty\om^{4-j}(\eta_{13})\om''(\eta_{23})\om^j(\eta)d\eta,\notag\\
&\lambda_{5,j(3)}=\frac{1}{a_{23}}\int_{-\infty}^\infty\om''(\eta_{13})\om^{4-j}(\eta_{23})\om^j(\eta)d\eta,\notag\\
&\lambda_{6,jkm}=\frac{1}{a_{23}}\int_{-\infty}^\infty\om^{4-j}(\eta_{13})\om^{j-k}(\eta_{23})\om^k(\eta)\om''(\eta_{m3})d\eta.\notag
\end{align}
\subsection{Normalization}
\begin{align}
&\ol{\mathfrak{R}}_{(1),2}=\sum_{i=1}^3
\theta_{i3}^{7/3}(\Lambda^{\,2}_i-1)+2\sum_{l,n}\theta_{l3}^{5/3}\theta_{n3}^{2/3}\,\Lambda_l\,\Lambda_n\lambda_{I,1}^{(0)}(\sigma_{ln}),\notag\\
&\mathfrak{K}=\frac{\sigma_{12}}{\beta_1}q_{11}^{(1)}
-\frac{\sigma_{23}}{\beta_2}\big(q_{31}^{(1)}+Q_3\big)+\mathfrak{K}_C,\;q^{(m)}_{ik}=\beta_i^2K_{ik}^{(2)}+(-1)^m\frac43\gamma^3K_{ik}^{(5)},\label{5A}\\
&\mathfrak{K}_C=
\sum_{l,n}\Big(C_{2,ln}^{(1)}-\frac{4}{3}\gamma^3C_{5,ln}\Big)-\frac{4}{3}\gamma^3C_{5,123}, \label{6AA}\\ &Q_3=\sum_{j=1}^2\Big(R_{2,j3}^{(1)}-\frac{4}{3}\gamma^3R_{5,j3}\Big)-\frac{4}{3}\gamma^3R_{5,123},\label{6A}\\
&p_{ik}=K_{ik}^{(2)}-\frac{r_2}{r_1}K_{ik}^{(1)}+R_i^{(k)},\,R_3^{(1)}=\sum_{j=1}^2R_{2,j3},\,R_i^{(k)}=0\;\text{if}\,i\neq3,\,k\neq1,\label{7AA}\\
&e_{i0}=-r_{1}q^{(2)}_{i0}+K_{i0}^{(1)}\sum_{j=1}^3q^{(2)}_{j0}-2r_1(\beta_i^2K_{i1}^{(2)}+\zeta_{13}+\zeta_{i})+2K_{i0}^{(1)}\Psi,\label{8A}\\
&e_{i1}=r_{1}(q_{i1}^{(1)}+Q_i)+K_{i1}^{(1)}\sum_{j=1}^3q^{(2)}_{j0},\quad \Psi=\sum_{i=1}^3\beta_i^2K_{i1}^{(2)}+2\sum_{ln}\zeta_{ln},\label{9A}\\
&\quad \zeta_{ln}=\beta_l^2\beta_nK_{l}^{(1)}K_{n}^{(1)}\lambda_{I1}^{(0)}(\sigma_{ln}),\quad\zeta_{1}=\zeta_{12},\quad \zeta_{3}=\zeta_{23},\quad Q_1=0.\label{11A}
\end{align}
\subsection{Asymptotic analysis}
\begin{align}
&\tilde{E}_{11}=-\ol{r}_2+\ol{r}_1\theta_{13}^{1/3}-\ol{r}_2\mu^{3(1+\alpha)/2}z'(\sigma_{12})
-2\ol{r}_1\mu^{5(1+\alpha)/2}\lambda'_{21}(\sigma_{12})\notag\\
&+2\ol{r}_1\mu^{3+\alpha}\Big(\big(\sigma\lambda_{21}(\sigma+\theta_{12}\sigma_{23})\big)'_\sigma\big|_{\sigma=\sigma_{12}}-\Lambda_{1}z'(\sigma_{13})\Big)+
O_\lambda(\mu^{3+2\alpha}),\notag\\
&\tilde{E}_{12}=\ol{r}_1-\ol{r}_2\theta_{13}^{1/3}-2\ol{r}_1\mu^{3+\alpha}\big(\lambda_{21}(\sigma_{12})z'(\sigma_{23})-\Lambda_{1}z'(\sigma_{13})\big)\notag\\
&+\ol{r}_2\mu^{(7+3\alpha)/2}z'(\sigma_{23})+O_\lambda(\mu^{4}),\quad\ol{r}_1=\sum_{i=1}^3\theta_{1i}^{1/3},\quad \ol{r}_2=\sum_{i=1}^3\theta_{i3}^{1/3}\notag\\
&\tilde{E}_{21}=\frac73+\frac73\mu^{3(1+\alpha)/2}z'(\sigma_{12})-4\mu^{2}\lambda_{21}(\sigma_{23})+O_\lambda(\mu^{3+2\alpha}),\notag\\
&\tilde{E}_{22}=-\frac73\ol{r}_1+\frac73\mu^{(3+\alpha)/2}
+4\mu^{2}\big(\lambda_{21}(\sigma_{23})-\frac53z'(\sigma_{23})\big)+O_\lambda(\mu^{3+\alpha}),\notag\\
&\tilde{F}_1=-\frac67\mu^2\Big\{\lambda_{21}(\sigma_{23})\big(1+\mu^{(1+\alpha)/2}\big)
+\mu^{(3+\alpha)/2}\big(2\lambda_{43}(\sigma_{23})\notag\\
&-\frac{11}{7}\lambda_{21}(\sigma_{23})\big)+O_\lambda(\mu^{2})\Big\},\quad
\tilde{F}_2=2\mu^2\Big\{\lambda_{21}(\sigma_{23})-\frac{56}{3}\mu^{1+\alpha}\big(\lambda_{21}(\sigma_{13})\notag\\
&-\lambda_{21}(\sigma_{12})\lambda_{21}(\sigma_{23})\big)
+\frac{28}{3}\mu^{(3+\alpha)/2}\big(\lambda_{43}(\sigma_{23})-\lambda_{21}(\sigma_{23})\big)+O_\lambda(\mu^{2})\Big\}.\notag
\end{align}

\end{document}